\documentclass[MSNbibl,number,citesort,dvips]{arxbj}
\usepackage{upgreek}
\usepackage{mathbh}
\usepackage{mathrsfs}

\aid{0}
\volume{18}
\issue{4}
\pubyear{2012}
\firstpage{1128}
\lastpage{1149}
\doi{10.3150/11-BEJ375}

\makeatletter
\renewcommand{\Re}{\operatorname{Re}}

\newcommand{\tilP}{P}
\newcommand{\tilPhi}{\Phi}
\newcommand{\Fourier}[1]{ \widehat{#1}}
\newcommand{\scalp}[2]{#1\cdot #2}

\newcommand{\R}{\mathbb R}
\newcommand{\Rd}{{\mathbb R^d}}
\newcommand{\I}{\mathbh{1}}
\newcommand{\N}{\mathbb{N}}
\newcommand{\var}{\operatorname{Var}}
\newcommand{\Bb}{\mathscr{B}}
\newcommand{\sfera}{ \mathbb{S}}
\newcommand{\eqref}[1]{(\ref{#1})}

\newtheorem{theorem}{Theorem}[section]

\newtheorem{proposition}[theorem]{Proposition}

\newremark{example}[theorem]{Example}
\newremark{remark}[theorem]{Remark}
\makeatother

\begin{document}
\begin{frontmatter}

\title{Coupling property and gradient estimates of L\'evy
processes via the symbol}

\runtitle{Coupling property and gradient estimates of L\'evy processes}

\begin{aug}
%%%% inicialai - be tarpu
\author[1]{\fnms{Ren\'{e} L.} \snm{Schilling}\thanksref{1}\ead[label=e1]{rene.schilling@tu-dresden.de}},
\author[2]{\fnms{Pawe{\l}} \snm{Sztonyk}\thanksref{2}\ead[label=e2]{sztonyk@pwr.wroc.pl}}
\and
\author[3]{\fnms{Jian} \snm{Wang}\corref{}\thanksref{1,3}\ead[label=e3]{jianwang@fjnu.edu.cn}}

\runauthor{R.L. Schilling, P. Sztonyk and J. Wang}

\address[1]{TU Dresden, Institut f\"{u}r Mathematische Stochastik,
01062 Dresden, Germany.\\ \printead{e1}}

\address[2]{Institute of Mathematics and Computer Science,
Wroc{\l}aw University of Technology, Wybrze{\.z}e Wyspia{\'n}\-skie\-go
27, 50-370 Wroc{\l}aw, Poland. \printead{e2}}

\address[3]{School of Mathematics and Computer Science, Fujian Normal
University, 350007 Fuzhou, P.R.~China. \printead{e3}}
\end{aug}

% HISTORY:
\received{\smonth{3} \syear{2011}} \revised{\smonth{4} \syear{2011}}

% ABSTRACT
%
\begin{abstract}
We derive explicitly the coupling property for the transition semigroup
of a L\'evy process and gradient estimates for the associated semigroup
of transition operators. This is based on the asymptotic behaviour of
the symbol or the characteristic exponent near zero and infinity,
respectively. Our results can be applied to a large class of L\'evy
processes, including stable L\'evy processes, layered stable processes,
tempered stable processes and relativistic stable processes.
\end{abstract}

% KEYWORDS
%
\begin{keyword}
\kwd{coupling} \kwd{gradient estimates} \kwd{L\'evy process}
\kwd{symbol}
\end{keyword}

\end{frontmatter}
\def\theequation{\arabic{section}.\arabic{equation}}
%s1 ###
\section{Introduction and main results}\label{section1}
Let $X_t$ be a pure jump L\'evy process on $\R^d$ with the symbol (or
characteristic exponent)
\[
\Phi(\xi) =\int_{z\neq0}
\bigl(1-\mathrm{e}^{\mathrm{i}\scalp{\xi}{z}}+\mathrm{i}\scalp{\xi}{z}\I
_{B(0,1)}(z)\bigr)\nu(\mathrm{d}z),
\]
where $\nu$ is the L\'evy measure, that is, a $\sigma$-finite measure
on $\R^d\setminus\{0\}$ such that the integral $\int_{z\neq0}(1\wedge
|z|^2)\nu(\mathrm{d}z)<\infty$. There are many papers studying regularity
properties of L\'evy processes in terms of the symbol $\Phi$. For
example, recently~\cite{KV}, Theorem 1, points out the relations between
the classic Hartman--Wintner condition (see~\cite{HW} or \eqref{den1}
below) and some smoothness properties of the transition density for
L\'evy processes. In particular, the condition that the symbol
$\Phi(\xi)$ of the L\'evy process $X_t$ satisfies
\begin{equation}\label{den1}
\liminf_{|\xi|\to\infty}\frac{\Re\Phi(\xi)}{\log(1+|\xi
|)}=\infty
\end{equation}
is \textit{equivalent} to the statement that for all $t>0$ the random
variables $X_t$ have a transition density $p_t(y)$ such that $\nabla
p_t\in L_1(\R^d)\cap C_\infty(\R^d)$, where $C_\infty(\R^d)$ denotes
the set of all continuous functions which vanish at infinity. The main
purpose of this paper is to derive an explicit coupling property and
gradient estimates of L\'evy processes directly from the corresponding
symbol $\Phi$.

Let $(X_t)_{t\geq0}$ be a Markov process on $\R^d$ with transition
probability function $\{P_t(x,\cdot)\}_{t\geq0, x\in\R^d}$. An
$\R^{2d}$-valued process $(X'_t,X''_t)_{t\geq0}$ is called \textit{a
coupling of the Markov process} $(X_t)_{t\geq0}$, if both
$(X'_t)_{t\geq0}$ and $(X''_t)_{t\geq0}$ are Markov processes which have
the same transition functions $P_t(x,\cdot)$ but possibly different
initial distributions. In this case, $(X'_t)_{t\geq0}$ and
$(X''_t)_{t\geq0}$ are called the \textit{marginal processes} of the
coupling process; the \textit{coupling time} is defined by
$T:=\inf\{t\geq0\dvt X'_t=X''_t\}$. The coupling $(X'_t,X''_t)_{t\geq0}$ is
said to be \textit{successful} if $T$ is a.s. finite. If for any two
initial distributions $\mu_1$ and $\mu_2$, there exists a successful
coupling with marginal processes starting from $\mu_1$ and $\mu_2$,
respectively, we say that $X_t$ has \textit{the coupling property} (or
admits \textit{successful couplings}). According to~\cite{Li} and the
proof of~\cite{RWJ}, Theorem 4.1, the coupling property is equivalent
to the statement that
\[%\label{prex1}
\lim_{t\to\infty}\|P_t(x,\cdot)-P_t(y,\cdot)\|_{\var}=0\qquad
\mbox{for any } x, y\in\R^d,
\]
where $P_t(x,\cdot)$ is the transition function of the Markov process
$(X_t)_{t\geq0}$. By $\|\mu\|_{\var}$ we denote the total variation norm
of the signed measure $\mu$. We know from~\cite{RWJ}, Theorem 4.1, that
every L\'evy process has the coupling property if the transition
functions have densities for all sufficiently large $t>0$. In this
case, the transition probability function satisfies
\begin{equation}\label{1proff111}
\|P_t(x,\cdot)-P_t(y,\cdot)\|_{\var}\leq\frac{C(1+|x-y|)}{\sqrt
{t}}\wedge2
\qquad\mbox{for }t>0
\mbox{ and }x, y\in\R^d.
\end{equation}
It is clear that for any $x, y\in\R^d$ and $t\geq0$,
$\|P_t(x,\cdot)-P_t(y,\cdot)\|_{\var}\leq2$, and that the norm
$\|P_t(x,\cdot)-P_t(y,\cdot)\|_{\var}$ is decreasing with respect to
$t$. This shows that it is enough to estimate
$\|P_t(x,\cdot)-P_t(y,\cdot)\|_{\var}$ for large values of $t$. We will
call any estimate for $\|P_t(x,\cdot)-P_t(y,\cdot)\|_{\var}$ an
\textit{estimate of the coupling time}. The rate $1/\sqrt{t}$ in
\eqref{1proff111} is not optimal for general L\'evy processes which
admit successful couplings. For example, for rotationally invariant
$\alpha$-stable L\'evy processes we can prove, see~\cite{BRW}, Example 2.3, that
\[
\|P_t(x,\cdot)-P_t(y,\cdot)\|_{\var}\asymp\frac{1}{t^{1/\alpha
}}\qquad\mbox{as }t\to\infty,
\]
where for any two non-negative functions $g$ and $h$, the notation
$g\asymp h$ means that there are two positive constants $c_1$ and $c_2$
such that $c_1g\leq h\leq c_2g$.

Let $P_t(x,\cdot)$ and $P_t$ be the transition function and the
semigroup of the L\'evy process $X_t$, respectively. We begin with
coupling time estimates of L\'evy processes which satisfy the following
Hartman--Wintner condition for some $t_0>0$:
\begin{equation}\label{coup1}
\liminf_{|\xi|\to\infty}\frac{\Re\Phi(\xi)}{\log(1+|\xi|)} >
\frac d{t_0};
\end{equation}
this condition actually ensures that the transition function of the
L\'evy process $X_t$ is, for all $t>t_0$, absolutely continuous, see, for exampe,~\cite{HW} or~\cite{KV}. Note that \eqref{coup1} becomes
\eqref{den1} if $t_0\to0$.

\begin{theorem}\label{th1}
Suppose that \eqref{coup1} holds and
\[
\Re\Phi(\xi)\asymp f(|\xi|)\qquad
\mbox{as } |\xi| \to0,
\]
where $f\dvtx [0,\infty)\to\R$ is a strictly increasing function which is
differentiable near zero and which satisfies
\[
\liminf_{r\to0}f(r)|\log r|<\infty
\]
and
\[%\label{ccoup1}
\limsup_{s\to0 }f^{-1}(2 s)/f^{-1}(s)<\infty.
\]
Then the corresponding L\'evy process $X_t$ has the coupling property,
and there exist two constants $c,t_1>0$ such that for any $x, y\in\R
^d$ and $t\geq t_1$,
\[
\|P_t(x,\cdot)-P_t(y,\cdot)\|_{\var}\leq c f^{-1}(1/t).
\]
\end{theorem}

It can be seen from the above remark on rotationally invariant
$\alpha$-stable L\'evy processes that the estimate in Theorem~\ref{th1}
is sharp.

\begin{remark}
(1) In our earlier paper~\cite{BRW}, in particular~\cite{BRW}, Theorem 1.1 and (1.3), we showed that the following condition on the L\'evy
measure $\nu$ ensures that a (pure jump) L\'evy process admits a
successful coupling:
\begin{equation}\label{brw-cond}
\nu(\mathrm{d}z) \geq|z|^{-d}g(|z|^{-2})\,\mathrm{d}z
\end{equation}
for some Bernstein function $g$. In the present paper, we use a
different condition in terms of the characteristic exponent $\Phi(\xi
)$. %In general, this condition and \eqref{brw-cond} cannot be
%compared; it depends very much on the situation which condition is
Let us briefly compare~\cite{BRW}, Theorem 1.1, and Theorem~\ref{th1}.
If \eqref{brw-cond} holds, then we know that
\[
\Phi(\xi)
=\Phi_{\rho}(\xi)+\Phi_{\mu}(\xi),
\]
where $\Phi_{\rho}$ and $\Phi_{\mu}$ denote the (pure-jump)
characteristic exponents with L\'{e}vy measures $\rho(\mathrm{d}z)=
|z|^{-d}g(|z|^{-2})\,\mathrm{d}z$ and $\mu=\nu-\rho$, respectively. Note that
\eqref{brw-cond} guarantees that $\mu$ is a nonnegative measure. By
\cite{JS}, Lemma 2.1, and some tedious, but otherwise routine,
calculations one can see that $\Phi_{\rho}(\xi)\asymp g(|\xi|^2)$
as $|\xi| \to0$.

If $g$ satisfies~\cite{BRW}, (2.10) and (2.11), -- these conditions
coincide with the asymptotic properties required of $f$ in Theorem \ref
{th1} --, we can apply Theorem~\ref{th1} to the symbol $\Phi_{\rho
}(\xi)$ with $f(s)=g(s^2)$, and follow the argument of~\cite{BRW}, Proposition 2.9 and Remark~2.10, to get a new proof of
\cite{BRW}, Theorem 1.1. Note that this argument uses the fact that, we can
(in law) decompose the L\'{e}vy process with exponent $\Phi(\xi)$
into two independent L\'{e}vy processes with characteristic exponents
$\Phi_{\rho}(\xi)$ and $\Phi_{\mu}(\xi)$, respectively.

(2) The considerations from (1) can be adapted to show that we may
replace the two-sided estimate $\Re\Phi(\xi)\asymp f(|\xi|)$ in
Theorem~\ref{th1} by $\Re\Phi(\xi)\geq cf(|\xi|)$; this, however,
requires that we know in advance that $\Phi(\xi)-cf(\xi)$ is a
characteristic exponent of some L\'{e}vy process. While this was
obvious under \eqref{brw-cond} and for the difference of two L\'evy
measures being again a nonnegative measure, there are no good
conditions in general when the difference of two characteristic
exponents is again an characteristic exponent of some L\'evy process.

(3) The present result, Theorem~\ref{th1}, trivially applies to most
subordinate stable L\'evy processes: here the characteristic exponent
is of the form $f(|\xi|^\alpha)$, $0<\alpha\leq2$, but the
corresponding L\'evy measures cannot be given in closed form. In
Example~\ref{ex1} below, we have a situation where the L\'evy measure
is known. Nevertheless, the methods of~\cite{BRW} are only applicable
in \textit{one} particular case, while Theorem~\ref{th1} applies to
\textit{all} non-degenerate settings.
\end{remark}

Now we turn to explicit gradient estimates for the semigroup of a
L\'evy process. For a function $u\in B_b(\R^d)$ we define
\[
|\nabla u(x)|
:=\limsup_{y\to x}\frac{|u(y)-u(x)|}{|y-x|},\qquad x\in\R^d.
\]
If $u$ is differentiable at $x$, then $|\nabla u(x)|$ is just the norm
of the gradient of $u$ at $x$. We are interested in sub-Markov
semigroups $P_t$ on $B_b(\R^d)$ which satisfy that for some positive
function $\phi$ on $(0,\infty)$
\[
\|\nabla P_t u\|_\infty
\leq\|u\|_\infty\phi(t),\qquad t>0, u\in B_b(\R^d).
\]
Similar uniform gradient estimates for Markov semigroups have attracted
a lot of attention in analysis and probability, for example, see
\cite{PW} and references therein. Because of the Markov property of the
semigroup $P_t$, $\phi(t)$ is decreasing with respect to $t$. Thus, it
is enough to obtain sharp estimates for $\phi(t)$ both as $t\to0$ and
$t\to\infty$. For L\'evy processes, we have the following theorem.
\begin{theorem}\label{th2}
Assume that \eqref{den1} holds. If there is a strictly increasing
function $f$ which is differentiable near infinity and which satisfies
\[
\limsup_{s\to\infty}f^{-1}(2 s)/f^{-1}(s)<\infty,
\]
and
\[
\Re\Phi(\xi)\asymp f(|\xi|)\qquad
\mbox{as } |\xi|\to\infty,
\]
then there exists a constant $c>0$ such that for $t>0$ small enough,
\begin{equation}\label{th21333}
\|\nabla P_t u\|_\infty\leq c\|u\|_\infty f^{-1}(1/t),\qquad
u\in{B}_b(\R^d).
\end{equation}

Similarly, let $f$ be a strictly increasing function which is
differentiable near zero and which satisfies
\[
\liminf_{r\to0}f(r)|\log r|<\infty,
\qquad
\limsup_{s\to0 }f^{-1}(2 s)/f^{-1}(s)<\infty
\]
and
\[
\Re\Phi(\xi)\asymp f(|\xi|)\qquad
\mbox{as } |\xi|\to0.
\]
Then there exists a constant $c>0$ such that \eqref{th21333} holds for
$t>0$ large enough.
\end{theorem}

We will see in Remark~\ref{stable} below that Theorem~\ref{th2} is also
sharp for rotationally invariant $\alpha$-stable L\'evy processes.
Roughly speaking, Theorems~\ref{th1} and~\ref{th2} show that the
gradient estimate~\eqref{th21333} for a L\'evy process for small $t\ll
1$ depends on the asymptotic behaviour of the symbol $\Phi$ near
infinity, while \eqref{th21333} for large $t\gg1$ relies on the
asymptotic behaviour of the symbol $\Phi$ near zero. This situation is
familiar from estimates of the coupling time of L\'evy processes. More
details can be found in the examples given below.

In order to illustrate the power of Theorems~\ref{th1} and~\ref{th2},
we present two examples.

\begin{example}\label{ex0}
Let $X_t$ be a subordinate Brownian motion with symbol $f(|\xi|^2)$,
where $f(\lambda)= \lambda^{\alpha/2} (\log(1+\lambda))^{\beta
/2}$, $\alpha\in(0,2)$ and $\beta\in(-\alpha,2-\alpha)$. To see
that $f$ is indeed a Bernstein function we observe that $\lambda$,
$\log(1+\lambda)$ and $\lambda/\log(1+\lambda)$ are complete
Bernstein functions, and that for $\alpha,\beta\geq0$
\[
\lambda^{\alpha/2} \cdot\bigl(\log(1+\lambda)\bigr)^{\beta/2}\qquad
\mbox{is a complete Bernstein function if }
\frac\alpha2+\frac\beta2\leq1,
\]
while for $-\alpha\leq\beta\leq0 \leq\alpha$
\[
\lambda^{(\alpha-\beta)/2} \cdot\biggl(\frac\lambda{\log(1+\lambda
)}\biggr)^{\beta/2}
\qquad
\mbox{is a complete Bernstein function if }
\frac\alpha2+\frac\beta2\leq1.
\]
This follows easily from~\cite{SSV}, (Proof of) Proposition 7.10, see
also~\cite{SV}, Examples 5.15, 5.16.

There are two constants $c_1,t_0>0$ such that for all $x, y\in\R^d$
and $t\geq t_0$,
\[
\|P_t(x,\cdot)-P_t(y,\cdot)\|_{\var}\leq c_1{t^{-1/(\alpha+\beta)}},
\]
and there exists a constant $c_2>0$ such that for all $u\in B_b(\R^d)$,
\[
\|\nabla P_tu\|_\infty
\leq\cases{
c_2\bigl[{t}^{-1}\bigl(\log(1+{t}^{-1})\bigr)^{-\beta/2}\bigr]^{1/\alpha} \|u\|_\infty
&\quad for small $t\ll1$;\cr
c_2t^{-1/(\alpha+\beta)}\|u\|_\infty
&\quad for large $t\gg1$.
}
\]
\end{example}

\begin{example}\label{ex1}
Let $\mu$ be a finite nonnegative measure on the unit sphere $\sfera$
and assume that $\mu$ is nondegenerate in the sense that its support is
not contained in any proper linear subspace of $\R^d$. Let
$\alpha\in(0,2)$, $\beta\in(0,\infty]$ and assume\vadjust{\goodbreak} that the L\'evy
measure $\nu$ satisfies that for some constant $r_0>0$ and any $A\in
\Bb(\R^d)$,
\[
\nu(A)
\geq\int_0^{r_0}\int_{\sfera}\I_A(s\theta)s^{-1-\alpha} \,\mathrm{d}s \mu
(\mathrm{d}\theta)
+ \int_{r_0}^\infty\int_{\sfera}\I_A(s\theta)s^{-1-\beta}\,\mathrm{d}s \mu
(\mathrm{d}\theta).
\]
Then, by Theorem~\ref{th1}, there are two constants $c_1,t_0>0$ such
that for all $x,y\in\R^d$ and $t\geq t_0$,
\[
\|P_t(x,\cdot)-P_t(y,\cdot)\|_{\var}\leq c_1{t^{-1/(\beta\wedge2)}}.
\]

In the present situation, the methods of~\cite{BRW} only apply if $\mu$
is (essentially) the uniform measure on $\sfera$; this is not the case
for the condition of Theorem~\ref{th1}.

Moreover, Theorem~\ref{th2} shows that there exists a constant $c_2>0$
such that for all $u\in B_b(\R^d)$,
\[
\|\nabla P_tu\|_\infty
\leq
\cases{
c_2\|u\|_\infty t^{-1/\alpha}
&\quad for small $t\ll1$;\cr
c_2\|u\|_\infty t^{-1/(\beta\wedge2)}
&\quad for large $t\gg1$.
}
\]
\end{example}

Coupling techniques for L\'evy-driven SDEs and L\'evy-type processes
have been considered in the literature before, see, for example,
\cite{AK,NK1,NK,Jian}. As far as we know, however, only the papers
\cite{wang1,W2} by F.-Y. Wang deal with couplings of L\'evy-driven
SDEs. Wang shows the existence of successful couplings and gradient
estimates for Ornstein--Uhlenbeck processes driven by L\'{e}vy
processes.

The remaining part of this paper is organized as follows. In Section
\ref{section2}, we first present estimates for the derivatives of the
density for infinitely divisible distributions in terms of the
corresponding L\'evy measure; this part is of some interest on its own.
Then we use these estimates to investigate derivatives of the density
for L\'{e}vy processes, whose L\'{e}vy measures have (modified) bounded
support. In Section~\ref{section3}, we give the proofs of all the
theorems and examples stated in Section~\ref{section1}, by using the
results of Section~\ref{section2}. Some remarks and examples are also
included here to illustrate the optimality and the efficiency of
Theorems~\ref{th1} and~\ref{th2}.

%s2 ###
\section{Derivatives of densities for infinitely divisible
distributions}\label{section2}

Let $\pi$ be an infinitely divisible distribution. It is well known
that its characteristic function $\Fourier{\pi}(\xi):=\int_{\R^d}
\mathrm{e}^{\mathrm{i}\scalp{\xi}{y}}\pi(\mathrm{d}y)$ is of the form $\exp(-\Phi(\xi))$, where
\[
\Phi(\xi) = \int_{y\neq0} \bigl(1-\mathrm{e}^{\mathrm{i}\scalp{\xi}{y}}+\mathrm{i}\scalp{\xi
}{y}\I_{B(0,1)}(y)\bigr)\nu(\mathrm{d}y),
\]
and $\nu$ is a L\'evy measure on $\Rd\setminus\{0\}$ such that
$\int_{y\neq0} (1\wedge|y|^2)\nu(\mathrm{d}y)<\infty$. In this section, we
first aim to study estimates for derivatives of the density of $\pi$.
As usual, we denote for every $n\in\N_0$ by $C_b^n(\R^d)$ the set of
all $n$-times continuously differentiable functions on $\R^d$ which
are, together with all their derivatives, bounded; for $n=0$ we use the
convention that $C_b^0(\R^d)=C_b(\R^d)$ denotes the set of continuous
and bounded functions on $\R^d$.

\begin{proposition}\label{lem:DESJ}
If for some $n$, $m\in\N_0$,
\begin{equation}\label{eq:Ass1}
\int \mathrm{e}^{-\Re\Phi(\xi)}(1+|\xi|)^{n+m}\,\mathrm{d}\xi<\infty,
\end{equation}
and
\begin{equation}\label{eq:Ass2}
\int_{|y|>1} |y|^{2\vee n} \nu(\mathrm{d}y)<\infty,
\end{equation}
then $\pi$ has a density $p\in C^{m+n}_b(\Rd)$ such that for every
$\beta\in\N_0^d$ with $|\beta|\leq m$,
\[
|\partial^\beta p(y)| \leq\psi(n,m,\nu) (1+|y|)^{-n},\qquad y\in\Rd,
\]
where for $n\geq0$
\[
\psi(n,m,\nu)
= C(n,d)\biggl(1+\int(|y|^2+|y|^{2\vee n}) \nu(\mathrm{d}y) \biggr)^{n} \int \mathrm{e}^{-\Re\Phi
(\xi)}(1+|\xi|)^{n+m}\,\mathrm{d}\xi.
\]
\end{proposition}

\begin{pf}
The existence of the density $p\in C^{m+n}_b(\Rd
)$ is a consequence of \eqref{eq:Ass1} and~\cite{SA}, Proposition 28.1, or~\cite{Pic}, Proposition 0.2.

To prove the second assertion, we recall some necessary facts and
notations. Given a function $f\in L^1(\R^d)$, its Fourier transform is
given by
\[
\Fourier{f}(\xi)=\int_{\R^d} f(y) \mathrm{e}^{\mathrm{i}\scalp{\xi}{y}}\,\mathrm{d}y.
\]
For $\xi\in\R^d$ and a multiindex $\beta=(\beta_1,\beta_2,\ldots,
\beta_d)\in\N_0^d$, we set $M_{\beta}(\xi):=\xi^\beta=
\xi_1^{\beta_1}\xi_2^{\beta_2}\cdots\xi_d^{\beta_d}$. If
$\Fourier{f}\in
C^N(\R^d)$ and $\partial^{\gamma}(M_\beta\Fourier{f})\in L^1(\Rd
)$ for
$N\in\N_0$ and every $\gamma\in\N_0^d$ such that $|\gamma|\leq N$,
then, using the inverse Fourier transform and the integration by parts
formula, we obtain that for every $\delta\in\N_0^d$ with $|\delta
|\leq
N$
\[
y^\delta \partial^\beta f(y)
= (2\uppi)^{-d} (-1)^{|\beta|} (\mathrm{i})^{|\beta|-|\delta|} \int\partial
^\delta[M_\beta\Fourier{f}](\xi) \mathrm{e}^{-\mathrm{i}\scalp{y}{\xi}} \,\mathrm{d}\xi.
\]
This yields
\begin{equation}\label{eq:F2}
|y^\delta\partial^\beta f(y)|
\leq(2\uppi)^{-d} \int| \partial^\delta[M_\beta\Fourier{f}](\xi)|
\,\mathrm{d}\xi.
\end{equation}
In particular, for every $n\in\N_0$,
\begin{equation}\label{eq:F3}
|y_k|^n |\partial^\beta f(y)|
\leq(2\uppi)^{-d} \int\biggl| \frac{\partial^n}{\partial\xi_k^n}
[M_\beta\Fourier{f}](\xi)\biggr|\, \mathrm{d}\xi.
\end{equation}

For $n=0$, the required assertion immediately follows from
\eqref{eq:F2} if we use $f=p$ and $\delta=0$. If $n>0$, then for every
$\beta\in\N_0^d$ such that $|\beta|=1$ we have
\[
\partial^\beta\Phi(\xi)
= -\mathrm{i} \int y^\beta\bigl(\mathrm{e}^{\mathrm{i}\scalp{\xi}{y}}-\I_{B(0,1)}(y)\bigr)\nu(\mathrm{d}y).
\]
By the H\"older inequality,
\begin{eqnarray}\label{eq:Diff1}
&&|\partial^\beta\Phi(\xi)|\nonumber\\
&&\quad\leq\biggl[\int|y|^2\nu(\mathrm{d}y)\biggr]^{1/2}
\biggl[2\int_{B(0,1)} \bigl(1-\cos(\scalp{\xi}{y})\bigr)\nu(\mathrm{d}y)
+ \nu(B(0,1)^c)\biggr]^{1/2} \nonumber\\
&&\quad\leq \biggl[ \int|y|^2\nu(\mathrm{d}y)\biggr]^{1/2}\biggl[|\xi|^2 \int_{B(0,1)} |y|^2 \nu(\mathrm{d}y)
+ \nu(B(0,1)^c)\biggr]^{1/2} \\
&&\quad\leq \int|y|^2\nu(\mathrm{d}y)\cdot(|\xi|^2 +1)^{1/2} \nonumber\\
&&\quad\leq (1+|\xi|) \int|y|^2\nu(\mathrm{d}y).\nonumber
\end{eqnarray}
On the other hand, for $1<|\beta|\leq n$, we have
\[
\partial^\beta\Phi(\xi) = -(\mathrm{i})^{|\beta|}\int y^\beta \mathrm{e}^{\mathrm{i}\scalp
{\xi}{y}}\nu(\mathrm{d}y),
\]
and so
\begin{equation}\label{eq:Diff2}
|\partial^\beta\Phi(\xi)| \leq\int|y|^{|\beta|} \nu(\mathrm{d}y).
\end{equation}
For symmetric L\'evy measures $\nu$ similar estimates are due to Hoh
\cite{Hoh}, see also~\cite{Jacob1}, Theorem~3.7.13.

Let $k\in\{1,\dots,d\}$ and $M\in\N$ with $M\leq n$. We use Faa di
Bruno's formula, see~\cite{ConSav96}, to obtain
\[
\frac{\partial^{M}}{\partial\xi_k^{M}} \Fourier{p}(\xi)
= M! \exp(-\Phi(\xi))
\sum_{j=1}^{M} \sum_{u(M,j)} \prod_{l=1}^{M}
\biggl(\frac{\partial^l (-\Phi)}{\partial\xi_k^l}(\xi)\biggr)^{\lambda
_l}\big/((\lambda_l!)(l!)^{\lambda_l}),
\]
where
\[
u(M,j)=\Biggl\{(\lambda_1,\dots,\lambda_{M})\dvt \lambda_l\in\N_0,
\sum_{l=1}^{M}\lambda_l=j, \sum_{l=1}^{M}l\lambda_l=M\Biggr\}.
\]
This, \eqref{eq:Diff1} and \eqref{eq:Diff2} yield
\begin{eqnarray*}
\biggl|\frac{\partial^{M}}{\partial\xi_k^{M}} \Fourier{p}(\xi)\biggr|
& \leq& |\exp(-\Phi(\xi))| \sum_{j=1}^{M} \sum_{u(M,j)}\prod_{l=1}^{M}
\frac{M!}{(\lambda_{l}!)(l!)^{\lambda_{l}}}
\biggl[(1+|\xi|)\int_{\Rd} |y|^{l\vee2}\nu(\mathrm{d}y)\biggr]^{\lambda_{l}} \\
& \leq& \mathrm{e}^{-\Re\Phi(\xi)} \sum_{j=1}^{M}
\biggl[(1+|\xi|)\int_{\Rd} (|y|^2+|y|^{2\vee n} )\nu(\mathrm{d}y)\biggr]^{j}
\sum_{u(M,j)}\prod_{l=1}^{M}
\frac{M!}{(\lambda_{l}!)(l!)^{\lambda_{l}}} \\
& \leq& \mathrm{e}^{-\Re\Phi(\xi)}(1+|\xi|)^M \sum_{j=1}^{M}
\biggl[\int_{\Rd} (|y|^2+|y|^{2\vee n} )\nu(\mathrm{d}y)\biggr]^{j}
\sum_{u(M,j)}\prod_{l=1}^{M}
\frac{M!}{(\lambda_{l}!)(l!)^{\lambda_{l}}} \\
& \leq& c_1(n) \mathrm{e}^{-\Re\Phi(\xi)}(1+|\xi|)^{n}
\biggl[1+\int_{\Rd} (|y|^2+|y|^{2\vee n} )\nu(\mathrm{d}y)\biggr]^{n}.
\end{eqnarray*}
We note that this inequality remains valid for $M=0$.

For $\beta\in\N_0^d$ with $|\beta|\leq m$, we can use the Leibniz rule
to get
\begin{eqnarray*}
\biggl|\frac{\partial^{n}}{\partial\xi_k^{n}} (M_\beta\Fourier{p})(\xi
) \biggr|
&\leq& \sum_{j=0}^{n} {n \choose j} \biggl| \frac{\partial^{j}}{\partial
\xi_k^{j}} M_\beta(\xi)
\frac{\partial^{n-j}}{\partial\xi_k^{n-j}} \Fourier{p}(\xi)\biggr| \\
&\leq& (1+|\xi|)^{|\beta|} \sum_{j=0}^{n} {n \choose j} \biggl|
\frac{\partial^{n-j}}{\partial\xi_k^{n-j}} \Fourier{p}(\xi)\biggr| \\
&\leq& c_2(n) \mathrm{e}^{-\Re\Phi(\xi)}(1+|\xi|)^{n+m}
\biggl[1+\int_{\Rd} (|y|^2+|y|^{2\vee n} )\nu(\mathrm{d}y)\biggr]^{n}.
\end{eqnarray*}
By \eqref{eq:F3}, we see
\[
|y_k|^{n} |\partial^\beta p(y)|
\leq c_2(n)\biggl[1+\int_{\Rd} (|y|^2+|y|^{2\vee n} )\nu(\mathrm{d}y)\biggr]^{n}
\int \mathrm{e}^{-\Re\Phi(\xi)}(1+|\xi|)^{n+m} \,\mathrm{d}\xi.
\]
Finally,
\[
(1+|y|)^{n} \leq2^{n-1}(1+|y|^{n}) \leq2^{n-1} d^{n/2} \Biggl(1+\sum
_{k=1}^d |y_k|^{n}\Biggr),
\]
and the required assertion follows with $C(n,d)=2^{n-1}
d^{n/2}(d+1)c_2(n)$.
\end{pf}

We will now study the derivatives of transition densities for L\'evy
processes with (modified) bounded support. For this, we need
Proposition~\ref{lem:DESJ}. Let $\Phi$ be the symbol (i.e., the
characteristic exponent) of a L\'evy process and consider for every
$r>0$ the semigroup of infinitely divisible measures $\{\tilP_t^r,
t\geq0\}$ whose Fourier transform is of the form
$\widehat{\tilP}^r_t(\xi)=\exp(-t\tilPhi_r(\xi))$, where
\[
\tilPhi_r(\xi) = \int_{|y|\leq r} (1-\mathrm{e}^{\mathrm{i}\scalp{\xi}{y}}+\mathrm{i}\scalp
{\xi}{y})\nu(\mathrm{d}y)
\]
($\nu$ is the L\'evy measure of the symbol $\Phi$). For $\rho>0$ and
$t>0$, we define
\begin{equation}\label{eq:rho}
\varphi(\rho)=\sup_{|\eta|\leq\rho}\Re\Phi(\eta)
\quad\mbox{and}\quad
h(t) := \frac1{\varphi^{-1}(1/t)}.
\end{equation}

\begin{proposition} \label{largetime} Assume that \eqref{coup1}
holds, and there exist
$m\in\N_0$ and $c,t_1>0$ such that for all $t\geq t_1$,
\begin{equation}\label{large3}
\int\exp(-(t\Re\Phi(\xi)))|\xi
|^{m}\,\mathrm{d}\xi\leq c\biggl(\varphi^{-1}\biggl(\frac{1}{t}\biggr)\biggr)^{m+d}.
\end{equation}
Then there is a constant $t_2=t_2(m,d)>0$ such that for any $t\geq t_2$,
there exists a density $p^{h(t)}_t\in C_b^{m}(\Rd)$ of
$\tilP^{h(t)}_t$, and for every $n\in\N_0$ and $\beta\in\N_0^d$ with
$|\beta|\leq m-n,$
\[
\bigl|\partial^\beta_y p^{h(t)}_t(y)\bigr| \leq C(m, n, {|\beta|},\Phi)
(\varphi^{-1}(t^{-1}))^{d+|\beta|} \bigl(1+\varphi^{-1}(t^{-1})|y|\bigr)^{-n},\qquad
y\in\Rd.
\]
\end{proposition}

\begin{pf} \textit{Step 1}.
For $\xi\in\R^d$,
\begin{eqnarray}\label{plargetime1}
|\widehat{\tilP}^r_t(\xi)|&=&\exp\biggl(-t\int_{|y|<r} \bigl(1-\cos({\scalp
{\xi}{y}})\bigr)\nu(\mathrm{d}y)\biggr)\nonumber\\
&=&\exp\biggl(-t\biggl(\Re\Phi(\xi)-\biggl(\int_{|y|\geq r} \bigl(1-\cos({\scalp{\xi
}{y}})\bigr)\nu(\mathrm{d}y)\biggr)\biggr)\biggr)\\
&\leq&\exp(-t(\Re\Phi(\xi)))\exp(2t\nu(B(0,r)^c)).\nonumber
\end{eqnarray}
By \eqref{coup1} and~\cite{SA}, Proposition 28.1, it follows that there
exists $t_3:=t_3(d)>0$ such that for all $r>0$ and for any $t\geq t_3$,
the measure $\tilP_t^r$ has a density $p^{r}_t\in C_b(\R^d)$.

\textit{Step 2.} For $t\geq t_3$, we define
$g_t(y)=h(t)^{d}p^{h(t)}_t(h(t) y)$. We consider the infinitely
divisible distribution $\pi_t(\mathrm{d}y)=g_t(y)\,\mathrm{d}y$. Its Fourier transform is
given by
\begin{eqnarray}\label{symbol}
\widehat{\pi}_t(\xi)
&=&(h(t))^{d}\int \mathrm{e}^{\mathrm{i}{\scalp{\xi}{y}}}p^{h(t)}_t(h(t) y)\,\mathrm{d}y\nonumber\\
&=&\int \mathrm{e}^{\mathrm{i} \scalp{\xi}{y}/h(t)}p^{h(t)}_t(y)\,\mathrm{d}y\nonumber\\
&=&\exp\biggl(-t\int_{|y|\leq h(t)}\biggl(1-\mathrm{e}^{\mathrm{i} \scalp{\xi}{y}/h(t)} +\frac
{\mathrm{i}\scalp{\xi}{y}}{h(t)}\biggr)\nu(\mathrm{d}y)\biggr)\\
&=&\exp\biggl(-\int_{|y|\leq1}(1-\mathrm{e}^{\mathrm{i} {\scalp{\xi}{y}}}+\mathrm{i}{\scalp{\xi
}{y}})\lambda_t(\mathrm{d}y)\biggr)\nonumber\\
&=&\exp(-G_t(\xi)),\nonumber
\end{eqnarray}
where $\lambda_t$ is the L\'evy measure of $\pi_t$, that is, for any Borel
set $B\subset\Rd\setminus\{0\}$,
\[
\lambda_t(B) = t\int_{|y|\leq h(t)} \I_B\bigl(y/h(t)\bigr)\nu(\mathrm{d}y).
\]
For $n\geq2$, we have
\begin{eqnarray*}
\int|y|^n \lambda_t(\mathrm{d}y)
&=& t\int_{|y|\leq h(t)} \biggl(\frac{|y|}{h(t)}\biggr)^n \nu(\mathrm{d}y)\\
&\leq& t\int_{|y|\leq h(t)} \biggl(\frac{|y|}{h(t)}\biggr)^2 \nu(\mathrm{d}y)\\
&\leq&2t \int_{|y|\leq h(t)} \frac{(|y|/h(t))^2}{1+(|y|/h(t))^2}\nu
(\mathrm{d}y)\\
&\leq& 2t \int\frac{(|y|/h(t))^2}{1+(|y|/h(t))^2}\nu(\mathrm{d}y)\\
&=&2t\int\!\!\!\int\bigl(1-\cos\bigl(y/h(t)\cdot\xi\bigr)\bigr)f_d(\xi)\,\mathrm{d}\xi\nu(\mathrm{d}y)\\
&=&2t\int\!\!\!\int\bigl(1-\cos\bigl(y\cdot\xi/h(t)\bigr)\bigr)\nu(\mathrm{d}y)f_d(\xi)\,\mathrm{d}\xi\\
&=& 2t\int\Re\Phi\biggl(\frac{\xi}{h(t)}\biggr) f_d(\xi)\,\mathrm{d}\xi,
\end{eqnarray*}
where
\[
f_d(\xi)=\frac{1}{2}\int_0^\infty(2\uppi\rho)^{-d/2}\mathrm{e}^{-|\xi
|^2/(2\rho)}\mathrm{e}^{-\rho/2}\,\mathrm{d}\rho.
\]
Obviously, $f_d(\xi)$ possesses all moments, see, for example,
\cite{Schill}, (2.5) and (2.6). By using several times the
subadditivity of $\eta\mapsto\sqrt{\Re\Phi(\eta)}$, we can easily find,
see, for example, the proof of~\cite{Schill2}, Lemma 2.3,
\[
\Re\Phi\biggl(\frac{\xi}{h(t)}\biggr)
\leq2(1+|\xi|^2) \sup_{|\eta|\leq1/h(t)} \Re\Phi(\eta) =
2(1+|\xi|^2) \frac1t.
\]
So,
\[
\int\Re\Phi\biggl(\frac{\xi}{h(t)}\biggr)f_d(\xi)\,\mathrm{d}\xi
\leq2 \sup_{|\eta|\leq1/h(t)} \Re\Phi(\eta) \int(1+|\xi
|^2)f_d(\xi)\,\mathrm{d}\xi
=: \frac{c_0}t.
\]
According to the definition of $h(t)$, we get that for any $t>0$,
\begin{equation}\label{largep1}
\int|y|^n \lambda_t(\mathrm{d}y)\leq2c_0.
\end{equation}

\textit{Step 3.}
It is easily seen from \eqref{symbol} that the characteristic exponent
of $\pi_t$ is $G_t(\xi)$, and
\[
\Re G_t(\xi)
= t\Re\bigl(\tilPhi_{h(t)}(h(t)^{-1}\xi)\bigr).
\]
Thus,
\begin{eqnarray*}
\Re G_t(\xi)
&=& t \biggl[\Re({\Phi}(h(t)^{-1}\xi)) - \int_{|y|>h(t)}\bigl(1-\cos({\scalp
{h(t)^{-1}\xi}{y}})\bigr)\nu(\mathrm{d}y) \biggr]\\
&\geq& t\Re({\Phi}(h(t)^{-1}\xi))- 2t \nu(B(0,h(t))^c).
\end{eqnarray*}
For any $t>0$,
\begin{eqnarray*}
\nu(B(0,h(t))^c)
&\leq&2\int_{|y|>h(t)} \frac{(|y|/h(t))^2}{1+|y|^2/h(t)^2}\nu(\mathrm{d}y)\\
&\leq&2 \int\frac{(|y|/h(t))^2}{1+(|y|/h(t))^2}\nu(\mathrm{d}y)\\
&=& 2 \int\!\!\!\int\bigl(1-\cos({\scalp{h(t)^{-1}\xi}{y}})\bigr)\nu(\mathrm{d}y) f_d(\xi)\,\mathrm{d}\xi
\\
&=& 2 \int\Re\Phi\biggl(\frac{\xi}{h(t)}\biggr) f_d(\xi)\,\mathrm{d}\xi\\
&\leq&2c_0\sup_{|\eta|\leq1/h(t)} \Re\Phi(\eta),
\end{eqnarray*}
where the last two lines follow from the same arguments as those
leading to \eqref{largep1}. Hence, for any $t>0$, we have
\[
t \nu(B(0,h(t))^c)
\leq2c_0t\sup_{|\eta|\leq1/h(t)} \Re\Phi(\eta)
= 2c_0.
\]
By \eqref{large3}, for $m\in\N_0$ and $c_1>0$, there exists
$t_4:=t_4(m,\Phi, c_1)\geq t_3$ such that for any $t\geq t_4$,
\[
\int\exp(-(t\Re\Phi(\xi)))|\xi|^m\,\mathrm{d}\xi
\leq c_1h(t)^{-(m+d)}.
\]
Therefore, we obtain
\begin{eqnarray}\label{largep2}
&&\int\exp[-\Re(G_t(\xi))]|\xi|^{m}\,\mathrm{d}\xi\nonumber\\
&&\quad\leq \mathrm{e}^{4c_0}\int\exp[-t\Re({\Phi}(\xi/h(t)))]|\xi|^m\,\mathrm{d}\xi\nonumber\\[-8pt]\\[-8pt]
&&\quad= \mathrm{e}^{4c_0}h(t)^{m+d}\int\exp[-(t\Re\Phi(\xi))]|\xi|^m\,\mathrm{d}\xi\nonumber\\
&&\quad=c_1\mathrm{e}^{4c_0}<\infty.\nonumber
\end{eqnarray}

\textit{Step 4.}
According to \eqref{largep1}, \eqref{largep2} and Proposition
\ref{lem:DESJ}, $g_t\in C_b^{m}(\Rd)$ for any $t\geq t_4$, and for every
$n\in\N_0$ and $\beta\in\N_0^d$ with $|\beta|\leq m-n$ we get
\[
|\partial^\beta_y g_t(y)|
\leq C(m, n, {|\beta|},\Phi) (1+|y|)^{-n},\qquad
y\in\Rd.
\]
This finishes the proof since $\partial^\beta_y
g_t(y)=h(t)^{d+|\beta|}\partial^\beta_yp^{h(t)}_t(h(t) y)$.
\end{pf}

The following result is the counterpart of Proposition~\ref{largetime},
which presents estimates for the derivatives of the densities
$p^{h(t)}_t$ for small time. Recall the definitions of $\varphi$
and $h$ from \eqref{eq:rho}.

\begin{proposition}\label{smalltime}
Assume that \eqref{den1} is satisfied, and there exist constants
$m\in\N_0$ and $c,t_0>0$ such that for every $0<t\leq t_0$,
\begin{equation}\label{samll3}
\int\exp[-(t\Re\Phi(\xi))] |\xi|^m\,\mathrm{d}\xi
\leq c\bigl(\varphi^{-1}(1/t)\bigr)^{m+d}.
\end{equation}
Then there is a constant $t_1>0$ such that for all $0<t\leq t_1$, there
exists a density $p^{h(t)}_t\in C_b^{m}(\Rd)$ of $\tilP^{h(t)}_t$.
Moreover, for every $n\in\N_0$ and $\beta\in\N_0^d$ with $|\beta
|\leq
m-n$,
\[
\bigl|\partial^\beta_y p^{h(t)}_t(y)\bigr|
\leq C(m, n, {|\beta|},\Phi, t_0) (\varphi^{-1}(t^{-1}))^{d+|\beta
|} \bigl(1+\varphi^{-1}(t^{-1})|y|\bigr)^{-n},\qquad y\in\Rd.
\]
\end{proposition}

\begin{pf}
The proof is similar to that of Proposition~\ref{largetime}, and we
only sketch some key differences. We continue to use the notations of
the proof of Proposition~\ref{largetime}. According to \eqref
{plargetime1} and \eqref{den1}, for all $r>0$ and $t>0$, the measure
$\tilP_t^r$ is absolutely continuous with respect to Lebesgue measure.
Since $t>0$ may be arbitrarily small, we need \eqref{den1} rather than
\eqref{coup1}.
Denote by $p^{r}_t$ its density. Following the argument of Proposition
\ref{largetime}, we find that \eqref{largep1} is still valid, and
according to \eqref{samll3}, there exists some $t_2>0$ such that
\eqref{largep2} holds for all $0<t\leq t_2$. The required assertion
follows now from Proposition~\ref{lem:DESJ}.
\end{pf}

%s3 ###
\section{Proofs of the main theorems and further examples}\label{section3}

We will now give the proofs for Theorem~\ref{th1} and~\ref{th2}. For
this, we need to estimate the coupling time of a general L\'evy
process. We will use the functions $\varphi$ and $h$ defined in~\eqref{eq:rho}.

\begin{theorem} \label{coup}
Assume that \eqref{coup1} holds, and there are $c,t_1>0$ such that for
any $t\geq t_1$,
\begin{equation}\label{coup2}
\int\exp[-(t\Re\Phi(\xi))] |\xi|^{d+2} \,\mathrm{d}\xi
\leq c(\varphi^{-1}(1/t))^{2d+2}.
\end{equation}
Then, the L\'evy process $X_t$ has the coupling property, and there
exist $t_2$, $C>0$ such that for any $x$, $y\in\R^d$ and $t\geq t_2$,
\begin{equation}\label{coup3}
\|P_t(x,\cdot)-P_t(y,\cdot)\|_{\var}
\leq C|x-y|\varphi^{-1}(1/t).
\end{equation}
%
%In particular, for any $x$, $y\in\R^d$,
%$$ \|P_t(x,\cdot)-P_t(y,\cdot)\|_{\var}
% = \mathsf{O}(\varphi^{-1}(1/t))
% \text{as\ \ }t\to\infty.$$
\end{theorem}

\begin{pf}
Set
\[
\tilPhi_r(\xi) = \int_{|y|\leq r} (1-\mathrm{e}^{\mathrm{i}\scalp{\xi}{y}}+\mathrm{i}\scalp
{\xi}{y})\nu(\mathrm{d}y)
\]
and
\[
\Psi_r(\xi)
:= \Phi(\xi)-\tilPhi_r(\xi)
=\int_{|y|>r}(1-\mathrm{e}^{\mathrm{i}\scalp{\xi}{y}})\nu(\mathrm{d}y)-\mathrm{i}\xi\cdot\int
_{1<|y|\leq r}y\nu(\mathrm{d}y).
\]
Let $Y_t$ and $Z_t$ be two independent L\'evy processes whose symbols
are $\tilPhi_{r}(\xi)$ and $\Psi_{r}(\xi)$, respectively.
Denote by $Q_{t}$ and $Q_{t}(x,\cdot)$ the semigroup and the
transition function of~$Y_t$. Similarly, $R_{t}$ and
$R_{t}(x,\cdot)$ stand for the semigroup and the transition function of
$Z_t$. Then,
\begin{eqnarray}\label{pcoup0}
\|P_t(x,\cdot)-P_t(y,\cdot)\|_{\var}
&=& \sup_{\|f\|_\infty\leq1}|P_tf(x)-P_tf(y)|\nonumber\\
&=& \sup_{\|f\|_\infty\leq1} |Q_t R_t f(x)-Q_t R_tf(y)|\nonumber\\[-8pt]\\[-8pt]
&\leq&\sup_{\|g\|_\infty\leq1} |Q_tg(x)-Q_tg(y)|\nonumber\\
&=& \|Q_t(x,\cdot)-Q_t(y,\cdot)\|_{\var}.\nonumber
\end{eqnarray}
Now, we take $r=h(t)$. Then, according to \eqref{coup2} and Proposition
\ref{largetime}, there exists $t_3>0$ such that for any $t\geq t_3$, the
kernel $Q_t$ has a density $q_t\in C^{d+2}_b(\Rd)$, and for all
$y\in\Rd$,
\begin{equation}\label{pcoup1}
|\nabla q_t(y)|
\leq c(d,\Phi)h(t)^{-(d+1)}\bigl(1+h(t)^{-1}|y|\bigr)^{-(d+1)}.
\end{equation}
Thus, for any $t\geq t_3$,
\begin{eqnarray}\label{pcoup2}
\|Q_t(x,\cdot)-Q_t(y,\cdot)\|_{\var}
&=&\sup_{\|f\|_\infty\leq1}\biggl|\int f(z)Q_t(x,\mathrm{d}z)-\int f(z)Q_t(y,\mathrm{d}z)\biggr|\nonumber\\
&=&\sup_{\|f\|_\infty\leq1} \biggl|\int f(z) q_t(z-x)\,\mathrm{d}z-\int
f(z)q_t(z-y)\,\mathrm{d}z\biggr|\\
&=&\int|q_t(z-x)- q_t(z-y)|\,\mathrm{d}z.\nonumber
\end{eqnarray}

Let $t\geq t_3$. Assume that $|x-y|>h(t)$. Then,
\[
\int|q_t(z-x)- q_t(z-y)|\,\mathrm{d}z
\leq2 \leq\frac{2|x-y|}{h(t)}.
\]

If $|x-y|\leq h(t)$, then, by \eqref{pcoup1},
\begin{eqnarray*}
&&\int|q_t(z-x) - q_t(z-y)|\,\mathrm{d}z\\
&&\quad=\int\limits_{|z-x|>2h(t)}|q_t(z-x)-q_t(z-y)|\,\mathrm{d}z
+ \int\limits_{|z-x|\leq2h(t)}|q_t(z-x)-q_t(z-y)|\,\mathrm{d}z\\
&&\quad\leq c(d,\Phi) \frac{|x-y|}{h(t)^{d+1}}
\biggl[\int\limits_{|z-x|>2h(t)}\Bigl[ h(t)^{d+1} \sup_{w\in B(z-x,
|y-x|)}|\nabla q _t(w)| \Bigr] \,\mathrm{d}z + \int\limits_{|z-x|\leq2h(t)}\mathrm{d}z \biggr]\\
&&\quad\leq c(d,\Phi)\frac{|x-y|}{h(t)^{d+1}}
\biggl[\int\limits_{|z-x|>2h(t)}\biggl[1+\frac{|z-x|}{2h(t)}\biggr]^{-d-1} \,\mathrm{d}z + c_d
(2h(t))^d \biggr]\\
&&\quad=c(d,\Phi) \frac{|x-y|}{h(t)^{d+1}} \int\biggl[1+\frac
{|z-x|}{2h(t)}\biggr]^{-d-1} \,\mathrm{d}z + 2^dc_dc(d,\Phi)\frac{|x-y|}{h(t)}\\
&&\quad\leq\frac{C|x-y|}{h(t)}.
\end{eqnarray*}
Therefore, there exists $C>0$ such that for all $x, y\in\R^d$ and
$t\geq
t_3$,
\begin{equation}\label{pcoup3}
\int|q_t(z-x)-q_t(z-y)|\,\mathrm{d}z
\leq\frac{C|x-y|}{h(t)}.
\end{equation}
The assertion follows now from \eqref{pcoup0}, \eqref{pcoup2} and
\eqref{pcoup3}.
\end{pf}

Next, we turn to the proof of Theorem~\ref{th1}.
\begin{pf*}{Proof of Theorem~\ref{th1}}
Under the conditions of Theorem~\ref{th1}, we can suppose that for any
\mbox{$\xi\in\R^d$},
\[
\Re\Phi(\xi)\geq F(|\xi|),
\]
where $F(r)$ is a strictly increasing and differentiable function on
$(0,\infty)$ such that
\[
F(r)=
\cases{
c_1f(r) &\quad if  $r\in(0,c_2)$;\cr
c_3\log(c_4+c_5 r) &\quad if $r\in[c_2,\infty)$
}
\]
for some constants $c_i>0$, $i=1,2,3,4,5$. Thus,
\begin{eqnarray*}
\int\exp[- t\Re\Phi(\xi)] |\xi|^{d+2} \,\mathrm{d}\xi
&\leq&\int\exp[-tF(|\xi|)] |\xi|^{d+2}\,\mathrm{d}\xi\\
&=& c_d \int_0^\infty \mathrm{e}^{-tr}[F^{-1}(r)]^{2d+1}\,\mathrm{d}F^{-1}(r)\\
&=& \frac{c_d}{2(d+1)}\int_0^\infty \mathrm{e}^{-tr}\,\mathrm{d}[F^{-1}(r)]^{2(d+1)}.
\end{eqnarray*}
Since $\liminf_{r\to0}f(r)|\log r|<\infty$ and $\limsup
_{s\to0
}f^{-1}(2 s)/f^{-1}(s)<\infty$, we have
\[
\liminf_{r\to0}F(r)|\log r|<\infty
\]
and
\[
\limsup_{s\to0 }F^{-1}(2 s)/F^{-1}(s)<\infty.
\]
Note that the function $F$ also satisfies
\[
\lim_{s\to\infty}F(s)/\log s=c_3.
\]
Then, following the proof of~\cite{BRW}, Theorem 2.1, we obtain that
for $t\to\infty$,
\begin{eqnarray*}
\int_0^\infty \mathrm{e}^{-tr}\,\mathrm{d}[F^{-1}(r)]^{2(d+1)}
&\asymp&[F^{-1}(1/t)]^{2(d+1)}\\
&=& [f^{-1}(1/t)]^{2(d+1)}\\
&\asymp&[\varphi^{-1}(1/t)]^{2(d+1)}.
\end{eqnarray*}
In the last step we used, in particular, the upper bound of the
two-sided comparison $\Re\Phi(\xi)\asymp f(|\xi|)$ as $|\xi|\to
0$. The
desired assertion follows from Theorem~\ref{coup}.
\end{pf*}

The following result is the short-time analogue of Theorem~\ref{coup}
which gives, additionally, gradient estimates for general L\'evy
processes.
\begin{theorem}\label{gradient}
Assume that \eqref{den1} holds and let $\phi$ and $h$ be as in \eqref
{eq:rho}. If there is a constant $c_0>0$ such that
\begin{equation}\label{gradi1}
\int\exp(-t\Re\Phi(\xi)) |\xi|^{d+2}\,\mathrm{d}\xi
\leq c_0(\varphi^{-1}(1/t))^{2d+2}\qquad
\mbox{for all } t\ll1\ [t\gg1],
\end{equation}
then there exists a constant $c>0$ such that for all $u\in B_b(\Rd)$
\begin{equation}\label{gradi2}
\|\nabla P_tu\|_\infty\leq c\|u\|_\infty\varphi^{-1}(1/t)\qquad
\mbox{for all } t\ll1\ [t\gg1].
\end{equation}
\end{theorem}

\begin{pf}
We will treat the short- and large-time cases separately.

Recall the notations used in the proof of Theorem~\ref{coup}: $Q_t$ and
$R_t$ are the semigroups corresponding to $\Phi_{r}(\xi)$ and
$\Psi_{r}(\xi)$, respectively. According to \eqref{gradi1} and
Proposition~\ref{smalltime}, for small enough $t\ll1$, and $r=h(t)$,
the measure $Q_t$ has a density $q_t\in C^{d+2}_b(\Rd)$ such that for
any $y\in\Rd$,
\begin{equation}\label{pgradi1}
|\nabla q_t(y)|
\leq c(d,\Phi) (\varphi^{-1}(t^{-1}))^{d+1} \bigl(1+\varphi
^{-1}(t^{-1})|y|\bigr)^{-(d+1)}.
\end{equation}
Then, for all $u\in B_b(\R^d)$,
\begin{eqnarray}\label{stableprocesses}
\sup_{\|u\|_\infty\leq1}\|\nabla Q_tu\|_\infty
&=&\sup_{\|u\|_\infty\leq1}\sup_{x\in\R^d}|\nabla Q_tu(x)|\nonumber\\
&=&\sup_{\|u\|_\infty\leq1}\sup_{x\in\R^d}\biggl|\nabla\int
q_t(z-x)\cdot u(z)\,\mathrm{d}z\biggr|\nonumber\\
%&=\sup_{\|u\|_\infty\leq1}\sup_{x\in\R^d}| \int\nabla q_t(z-x)\cdot
%u(z)dz|\\
&=&\sup_{x\in\R^d}\sup_{\|u\|_\infty\leq1}\biggl| \int\nabla
q_t(z-x)\cdot u(z)\,\mathrm{d}z\biggr|\nonumber\\[-8pt]\\[-8pt]
&=&\sup_{x\in\R^d} \int|\nabla q_t(z-x)|\,\mathrm{d}z\nonumber\\
&=&\int|\nabla q_t(z)|\,\mathrm{d}z\nonumber\\
&\leq& c\varphi^{-1}(t^{-1}),\nonumber
\end{eqnarray}
where we used \eqref{pgradi1} and dominated convergence. This
calculation shows
\[
\|\nabla Q_t u\|_\infty
\leq c\varphi^{-1}(t^{-1})\|u\|_\infty.
\]
Therefore,
\begin{equation}\label{proofex1}
\|\nabla P_t u\|_\infty
=\|\nabla Q_t(R_tu)\|_\infty
\leq c\varphi^{-1}(t^{-1})\|R_tu\|_\infty
\leq c\varphi^{-1}(t^{-1})\|u\|_\infty,
\end{equation}
which finishes the proof for small $t\ll1$.

If $t\gg1$ is sufficiently large, we can apply \eqref{coup3}, to find
for any $u\in B_b(\R^d)$ with $\|u\|_\infty= 1$,
\begin{eqnarray*}
|\nabla P_tu(x)|
&\leq&\limsup_{y\to x}\frac{|P_tu(x)-P_tu(y)|}{|y-x|}\\
&\leq&\limsup_{y\to x}\frac{\sup_{\|w\|_\infty\leq
1}|P_tw(x)-P_tw(y)|}{|y-x|}\\
&\leq&\limsup_{y\to x}\frac{\|P_t(x,\cdot)-P_t(y,\cdot)\|_{\var
}}{|y-x|}\\
&\leq& C\varphi^{-1}(t^{-1}).
\end{eqnarray*}
This finishes the proof for large $t\gg1$.
\end{pf}

\begin{remark}\label{stable}
Let $X_t$ be a rotationally invariant $\alpha$-stable L\'evy process on
$\R^d$, and $p_t$ be its density function. By the scaling property, for
any $t>0$ and $x\in\R^d$, $p_t(x)=t^{-d/\alpha}p_1(t^{-1/\alpha
}x)$. On
the other hand, it is well known that, see, for example,~\cite{Gl},
\[
|\nabla p_1(x)|\leq\frac{c}{1+|x|^{d+\alpha}}.
\]
Denote by $P_t$ the semigroup of $X_t$. Then, according to the proof of
\eqref{stableprocesses}, we have
\[
\sup_{\|u\|_\infty\leq1}\|\nabla{P}_tu\|_\infty=t^{-1/\alpha}\int
|\nabla{p}_1(z)|\,\mathrm{d}z.
\]
This implies that Theorem~\ref{th2} is optimal.
\end{remark}

We can now use Theorem~\ref{gradient} to prove Theorem~\ref{th2}.

\begin{pf*}{Proof of Theorem~\ref{th2}}
The second assertion easily follows from Theorem~\ref{th1} and the
proof of Theorem~\ref{gradient}. It is therefore enough to consider the
first conclusion. Under the conditions assumed in Theorem~\ref{th2}, we
know that for any $\xi\in\R^d$,
\[
\Re\Phi(\xi)\geq F(|\xi|),
\]
where $F(r)$ is an increasing function on $(0,\infty)$ such that
\[
F(r)=
\cases{ 0 &\quad if  $r\in(0,c_1]$;\cr
c_2f(r) &\quad if  $r\in(c_1,\infty)$
}
\]
for some constants $c_i>0$, $i=1,2$, and $f$ is strictly increasing and
differentiable on $(c_1,\infty)$. Therefore,
\[
\int\exp[-t\Re\Phi(\xi)] |\xi|^{d+2} \,\mathrm{d}\xi
\leq\int\exp[-tF(|\xi|)] |\xi|^{d+2} \,\mathrm{d}\xi.
\]

Since $\limsup_{s\to\infty}f^{-1}(2s)/f^{-1}(s)<\infty$, we can choose
$c>2$ and $s_0>0$ such that we have $f^{-1}(2s)\leq cf^{-1}(s)$ for all
$s\geq s_0$. For any $k\geq1$ the monotonicity of $f^{-1}$ gives
\[
f^{-1}(2^ks)
\leq c^kf^{-1}(s)
= 2^{k\alpha} f^{-1}(s),
\]
where we use $\alpha=\log_2c$. Then, for sufficiently small $t\ll1$,
\begin{eqnarray*}
&&\int\exp[- tF(|\xi|)] |\xi|^{d+2} \,\mathrm{d}\xi\\
&&\quad=\int_{|\xi|<c_1} |\xi|^{d+2} \,\mathrm{d}\xi
+ c_d \int_{c_1}^\infty \mathrm{e}^{-c_2 t f(r)} r^{2d+1} \,\mathrm{d}r\\
&&\quad\leq C_1 + c_d\int_0^\infty \mathrm{e}^{-c_2s} \mathrm{d}_s[f^{-1}(s/t)]^{2d+2}\\
&&\quad\leq C_1 + c_d\Biggl\{ \int_0^1 + \sum_{n=1}^\infty\int_{2^{n-1}}^{2^n}
\Biggr\} \mathrm{e}^{-c_2s} \,\mathrm{d}_s[f^{-1}(s/t)]^{2d+2}\\
&&\quad\leq
C_1 + c_d [f^{-1}(1/t)]^{2d+2} + c_d \sum_{n=1}^\infty\exp[-c_2
2^{n-1}] [ f^{-1}(2^n / t)]^{2d+2}\\
&&\quad\leq C_1 + c_d\Biggl( 1 + \sum_{n=1}^\infty\exp[ -c_2 2^{n-1}] 2^{n\alpha
(2d+2)}\Biggr) [f^{-1}(1/t)]^{2d+2}\\
&&\quad\leq C_1 + C_2 [f^{-1}(1/t)]^{2d+2}.
\end{eqnarray*}
Because of \eqref{den1} and $\Re\Phi(\xi)\asymp f(|\xi|)$ as
$|\xi|\to\infty$, we find $f^{-1}(1/t)\to\infty$ as $t\to0$. Thus,
\begin{eqnarray*}
\int\exp[-t\Re\Phi(\xi)] |\xi|^{d+2}\,\mathrm{d}\xi
&\leq&\int\exp[-tF(|\xi|)] |\xi|^{d+2} \,\mathrm{d}\xi\\
&\leq& C_3 [f^{-1}(1/t)]^{2d+2}.
\end{eqnarray*}
In the last step we used, in particular, the upper bound of the
two-sided comparison $\Re\Phi(\xi)\asymp f(|\xi|)$ as $|\xi|\to
\infty$. Now the assertion follows from Theorem~\ref{gradient}.
\end{pf*}

\begin{remark}
If we assume in the statement of Theorem~\ref{th2} that
$f^{-1}(s)=s^\alpha\ell(s)$ for some $\alpha>0$ and some positive
function $\ell$ which is slowly varying at $\infty$ -- that is,
$\lim_{s\to\infty} \ell(\lambda s)/\ell(s) = 1$ for every $\lambda>0$,
then standard Abelian and Tauberian arguments (see, e.g.,
\cite{BSW}, Theorems 1.7.1 and 1.7.1$'$, or~\cite{Feller}, Chapter XIII.5, Theorems 1 and 3), we can obtain that
\[
\int\exp[- tF(|\xi|)] |\xi|^{d+2} \,\mathrm{d}\xi
\asymp
[f^{-1}(1/t)]^{2d+2}
\qquad\mbox{as } t\to0.
\]
\end{remark}

Let us finally turn to the examples from Section~\ref{section1}.
\begin{pf*}{Proof of Example~\ref{ex0}}
The symbol of the subordinate Brownian motion here satisfies
\[
\Re\Phi(\xi) \asymp|\xi|^{\alpha+\beta} \qquad\mbox{as } |\xi|\to0,
\]
and
\[
 \Re\Phi(\xi)\asymp|\xi|^{\alpha}\bigl(\log(1+|\xi
|)\bigr)^{\beta/2} \qquad\mbox{as } |\xi|\to\infty.
\]
For $r>0$, set $f(r)=r^\alpha(\log(1+r))^{\beta/2}$ and $g(r)=(r
(\log(1+r))^{-\beta/2})^{1/\alpha}$. Then, for $r\to\infty$, we have
\begin{eqnarray*}
f(g(r))
&=& r \bigl(\log(1+r)\bigr)^{-\beta/2} \bigl[\log\bigl(1+\bigl(r\bigl(\log(1+r)\bigr)^{-\beta
/2}\bigr)^{1/\alpha} \bigr)\bigr]^{\beta/2}\\
&\asymp& r (\log r)^{-\beta/2} \biggl[\frac{2\log r -\beta\log\log
r}{\alpha}\biggr]^{\beta/2}\\
&=& r \biggl[\frac{2\log r -\beta\log\log r}{\alpha\log r}\biggr]^{\beta/2}\\
&\asymp& r.
\end{eqnarray*}
This shows that $f^{-1}(r)\asymp g(r)$ for $r\to\infty$, and now
Theorems~\ref{th1} and~\ref{th2} apply.
\end{pf*}

\begin{pf*}{Proof of Example~\ref{ex1}}
Let $Y_t$ and $Z_t$ be L\'evy processes whose L\'evy measures are given
by
\[
\nu^Y(A):=\int_0^{r_0}\int_{\sfera} \I_A(s\theta) s^{-1-\alpha
}\,\mathrm{d}s\mu(\mathrm{d}\theta)
+\int_{r_0}^\infty\int_{\sfera}\I_A(s\theta)s^{-1-\beta}\,\mathrm{d}s\mu
(\mathrm{d}\theta)
\]
and
\[
\nu^Z(\mathrm{d}z):=\nu(\mathrm{d}z)-\nu^Y(\mathrm{d}z)\geq0,
\]
respectively. After some elementary calculations, we see that the
symbol $\Phi^Y$ of $Y_t$ satisfies $\Re\Phi^Y(\xi)\asymp|\xi
|^\alpha$
as $|\xi|\to\infty$ and $\Re\Phi^Y(\xi)\asymp|\xi|^{\beta
\wedge2}$ as
$|\xi|\to0$. Let $P^Y_t(x,\cdot)$ and $P^Y_t$ denote the transition
function and the semigroup of $Y_t$. According to Theorems~\ref{th1}
and~\ref{th2}, we can prove the claim first for $Y$ (if we replace in
these Theorems $\Phi$, $P_t$ and $P_t(x,\cdot)$ by the corresponding
objects $\Phi^Y$, $P_t^Y$ and $P_t^Y(x,\cdot)$). To come back to the
original process resp. semigroup, we can now use \eqref{pcoup0} and
\eqref{proofex1}.
\end{pf*}

Example~\ref{ex1} applies to a large number of interesting and
important L\'evy processes, whose L\'evy measures are of the following
polar coordinates form:
\[
\nu(A)
= \int_0^\infty\int_{\sfera}\I_A(s\theta) Q(\theta,s)\,\mathrm{d}s \mu
(\mathrm{d}\theta);
\]
$Q(\theta,s)$ is a nonnegative function on $\sfera\times(0,\infty)$.
For instance, Example~\ref{ex1} is applicable for the following
processes:
\begin{enumerate}
\item Stable L\'evy processes~\cite{Sz1}:
\[
Q(\theta,s)\asymp s^{-\alpha-1},
\]
where $\alpha\in(0,2)$.
\item Layered stable processes~\cite{HO}:
\[
Q(\theta,s)\asymp s^{-\alpha-1}\I_{(0,1]}(s)+s^{-\beta-1}\I
_{[1,\infty)}(s),
\]
where $\alpha\in(0,2)$ and $\beta\in(0,\infty)$.
\item Tempered stable processes~\cite{RO}:
\[
Q(\theta,s)\asymp s^{-\alpha-1}\mathrm{e}^{-cs},
\]
where $\alpha\in(0,2)$ and $c>0$.
\item Relativistic stable processes~\cite{CM,BMR}:
\[
Q(\theta,s)\asymp s^{-\alpha-1}(1+s)^{(d+\alpha-1)/2}\mathrm{e}^{-s},
\]
where $\alpha\in(0,2)$.
\item Lamperti stable processes~\cite{CA}:
\[
Q(\theta,s) = s^{-\alpha-1}\frac{\exp(sf(\theta))s^{1+\alpha
}}{(\mathrm{e}^s-1)^{1+\alpha}},
\]
where $\alpha\in(0,2)$ and $f\dvtx \sfera\to\R$ such that $\sup
_{\theta\in\sfera}f(\theta)<1+\alpha$.
\item Truncated stable processes~\cite{KS}:
\[
Q(\theta,s)\asymp s^{-\alpha-1}\I_{(0,1]}(s),
\]
where $\alpha\in(0,2)$.\vadjust{\goodbreak}
\end{enumerate}

Motivated by Example~\ref{ex1}, we can present a short proof of (one
part of) F.-Y. Wang's result on explicit gradient estimates for the
semigroup of a general L\'evy processes.
\begin{theorem}[(F.-Y. Wang~\cite{W2}, Theorem 1.1)]\label{wang}
Let $X_t$ be a L\'evy process on $\R^d$ with L\'evy measure~$\nu$.
Assume that there exists some $r\in(0,\infty]$ such that
\[\label{1proff2}
\nu(\mathrm{d}z)\geq|z|^{-d}f(|z|^{-2})\I_{\{|z|\leq r\}}\,\mathrm{d}z,
\]
where $f$ is Bernstein function such that
\[
\liminf_{r\to\infty} \frac{f(r)}{\log r}=\infty
\quad\mbox{and}\quad
\limsup_{s\to\infty} \frac{f^{-1}(2 s)}{f^{-1}(s)}\in(0,\infty).
\]
Then there exists a constant $c>0$ such that for any $t>0$,
\[\label{th21}
\|\nabla P_tu\|_\infty
\leq c\|u\|_\infty f^{-1}\biggl(\frac{1}{t\wedge1}\biggr),\qquad u\in{B}_b(\R^d).
\]
\end{theorem}

\begin{pf}
According to the proofs of Example~\ref{ex1} and~\cite{W2}, Theorem~1.1, we see that $X_t$ can be decomposed into two independent
L\'evy processes $Y_t$ and $Z_t$, such that the symbol $\Phi^Y(\xi)$ of
$Y_t$ satisfies $\Phi^Y(\xi)\asymp f(|\xi|^{2})$ as $|\xi|\to
\infty$.
Now we can apply Theorem~\ref{th2} and the claim follows.
\end{pf}

\section*{Acknowledgements}
The authors would like to thank the referee and the
Associate Editor for their helpful
comments on earlier versions of the paper. Financial support through
DFG (Grant Schi 419/5-1) and DAAD (PPP Kroatien) (for Ren\'{e} L.
Schilling),
MNiSW, Poland (Grant N N201 397137) (for Pawe{\l} Sztonyk) and the
Alexander-von-Humboldt Foundation and the Natural Science
Foundation of Fujian (No. 2010J05002$)$ (for Jian Wang, the
corresponding author) is gratefully acknowledged.

%

% imsref loaded by arune.pranskunaite, 2011-10-25 10:25:01
%

\printhistory

\end{document}